\documentclass[11pt]{article}
\usepackage[dvips]{epsfig}
\usepackage{amsmath,amssymb,euscript,graphicx,amsfonts,verbatim}
\usepackage{enumerate,color}
\usepackage{graphicx}
\usepackage{enumitem}
\usepackage{enumerate,color}
\usepackage{graphicx}

\usepackage[colorlinks=true,linkcolor=blue,urlcolor=blue,citecolor=blue]{hyperref}

\newtheorem{theorem}{Theorem}[section]
\newtheorem{proposition}[theorem]{Proposition}
\newtheorem{lemma}[theorem]{Lemma}
\newtheorem{corollary}[theorem]{Corollary}

\newtheorem{definition}[theorem]{Definition}

\newtheorem{remark}[theorem]{Remark}
\newtheorem{notation}[theorem]{Notation}

\newenvironment{proof}{{\noindent \sc Proof.}}{\hfill $\Qed$\\}

\newcommand{\Qed}{\rule{2.5mm}{3mm}}

\newcommand{\Ga}{\Gamma}

\newcommand{\CC}{\mathbb{C}}

\DeclareMathOperator{\mat}{Mat}

\newcommand{\MX}{\mat_X(\CC)}
\newcommand{\e}{E^{*}}
\newcommand{\R}{{\cal R}}

\newcounter{case}

\renewcommand{\thecase}{\arabic{case}}

\newcounter{subcase}

\numberwithin{subcase}{case}

\begin{document}


\begin{center}
{\bf\Large Distance-regular graphs with classical parameters that support a uniform structure: case $q \le 1$} \\ [+4ex]
Blas Fernández{\small$^{a,b}$},  
Roghayeh Maleki{\small$^{a,b}$},  
\v Stefko Miklavi\v c{\small$^{a, b, c}$},
\\
Giusy Monzillo{\small$^{a, b,}$\footnote{Corresponding author e-mail: giusy.monzillo@famnit.upr.si}} 
\\ [+2ex]
{\it \small 
$^a$University of Primorska, UP IAM, Muzejski trg 2, 6000 Koper, Slovenia\\
$^b$University of Primorska, UP FAMNIT, Glagolja\v ska 8, 6000 Koper, Slovenia\\
$^c$IMFM, Jadranska 19, 1000 Ljubljana, Slovenia}
\end{center}


\begin{abstract}
Let $\Ga=(X,\mathcal{R})$ denote a finite, simple, connected, and undirected non-bipartite graph with vertex set $X$ and edge set $\mathcal{R}$. Fix a vertex $x \in X$, and define $\mathcal{R}_f = \mathcal{R} \setminus \{yz \mid \partial(x,y) = \partial(x,z)\}$, where $\partial$ denotes the path-length distance in $\Ga$. Observe that the graph $\Ga_f=(X,\mathcal{R}_f)$ is bipartite. We say that $\Ga$ supports a uniform structure with respect to $x$  whenever $\Ga_f$ has a uniform structure with respect to $x$.

Assume  that $\Ga$ is a distance-regular graph with classical parameters $(D,q,\alpha,\beta)$ with $q \le 1$. Recall that $q$ is an integer, which is not equal to $0$ or $-1$. The purpose of this paper is to study when $\Ga$ supports a uniform structure with respect to $x$.  The main result of the paper is a complete classification  of graphs with classical parameters with $q\leq 1$ and  $D \ge 4$ that support a uniform structure with respect to $x$.
\end{abstract}

\begin{quotation}

\end{quotation}

\begin{quotation}
\noindent 
{\em Math. Subj. Class.:}  05E99, 05C50.
\end{quotation}


\section{Introduction}
\label{sec:intro}

The notion of a uniform poset was introduced in 1990 by Terwilliger \cite{OldTer}. Vaguely speaking, a graded poset $P$ is uniform if the raising matrix and the lowering matrix of $P$ satisfy certain linear dependencies. In \cite{OldTer}, Terwilliger described the algebraic structure of the uniform posets and displayed eleven infinite families of examples. 
The notion of uniform poset could be easily adopted by bipartite graphs, as bipartite graphs can be viewed as Hasse diagrams for graded posets.  Indeed, let $\Ga$ denote a finite, simple, connected, and undirected bipartite graph with vertex set $X$. Fix a vertex $x \in X$, and define the partial order $\le$ on $X$ as follows. For vertices $y,z$ of $\Ga$ let $y \le z$ whenever $\partial(x,y)+\partial(y,z)=\partial(x,z)$, where $\partial$ denotes the path-length distance function in $\Ga$. 

Let $A$ denote the adjacency matrix of $\Ga$. The above poset structure on $X$ induces a decomposition $A=L+R$, where $L=L(x)$ (resp. $R=R(x)$) is the lowering matrix (resp. raising matrix) of $\Ga$ with 
respect to $x$. For vertices $y,z$ of $\Ga$ the $(y,z)$-entry of $L$ is $1$ if $z$ covers $y$, and $0$ otherwise. The matrix $R$ is the transpose of $L$. We say that $\Ga$ has a {\em uniform structure with respect to} $x$ if a certain
linear dependency among $RL^2$, $LRL$, $L^2R$, $L$ is satisfied (see Section \ref{sec:uniform} for more details). Uniform structures of $Q$-polynomial bipartite distance-regular graphs were studied in detail in \cite{MikTer}, where it was shown that except for one special case a uniform structure is always attained.

Assume now that $\Ga=(X,\mathcal{R})$ is a non-bipartite graph with vertex set $X$ and edge set $\mathcal{R}$. Fix a vertex $x \in X$, and define $\mathcal{R}_f = \mathcal{R} \setminus \{yz \mid \partial(x,y) = \partial(x,z)\}$. Observe that the graph $\Ga_f=(X,\mathcal{R}_f)$ is bipartite. We say that $\Ga$ {\em supports a uniform structure with respect to} $x$ if $\Ga_f$ has a uniform structure with respect to $x$. The study of non-bipartite graphs that support a uniform structure was initiated by Hou, Hou and Gao in \cite{hou2018folded}, where they determined conditions for which a folded $(2D+1)$-hypercube supports a uniform structure. 

The main purpose of this paper is to investigate which non-bipartite $Q$-polynomial distance-regular graphs support a uniform structure. We will concentrate on distance-regular graphs with classical parameters (which are also $Q$-polynomial). Let $\Ga$ denote a distance-regular graph with classical parameters $(D,q,\alpha,\beta)$ with $D \ge 3$. Recall that $q$ is an integer, different from $0$ and $-1$. The graph $\Ga$ is said to be of {\em negative type} whenever $q \le -2$. Our first main result concerns graphs of negative type. We show that if $\Ga$ is of negative type with $D \ge 4$, then $\Ga$ supports a uniform structure with respect to $x$ if and only if $\Ga$ is the dual polar graph $^2 A_{2D-1}(-q)$, see Theorem \ref{thm:main_neg_type}. Our second main result concerns graphs with classical parameters with $q=1$. We give a complete classification of these graphs that support a uniform structure with respect to $x$, see Theorem \ref{thm:main}. We will study graphs with classical parameters with $q \ge 2$ in our future paper. 

This paper is organized as follows. In Section \ref{sec:ter}, we recall the definition and some basic properties of the Terwilliger algebra $T$ of a graph $\Ga$. In Section \ref{sec:uniform}, we give a formal definition of a uniform structure for bipartite graphs. We also prove some auxiliary results. We then discuss distance-regular graphs, the $Q$-polynomial property,  and irreducible $T$-modules of a distance-regular graph in Sections  \ref{sec:drg} and \ref{endp1}. In Sections \ref{sec3}, \ref{sec:a1=0}, and \ref{sec:rnp} we consider  distance-regular graphs of negative type and prove our first main result. Finally, in Section \ref{sec:q=1} we consider  distance-regular graphs with classical parameters $(D,1,\alpha,\beta)$ and prove our second main result.


\section{Terwilliger algebra of a graph}
\label{sec:ter}

Throughout this paper, all graphs considered will be finite, simple, connected, and undirected. Let $\Ga=(X,\mathcal{R})$ denote a graph with vertex set $X$ and edge set $\mathcal{R}$. Let $\partial$ denote the path-length distance function of $\Ga$. The diameter $D$ of $\Ga$ is defined as $\max\{\partial(x,y) \mid x,y \in X \}$. In this section, we recall the definition of the Terwilliger algebra of $\Ga$. Let $\CC$ denote the field of complex numbers. Let $\MX$ denote the $\CC$-algebra consisting of all matrices whose rows and columns are indexed by $X$ and whose entries are in $\CC$. Let $V$ denote the vector space over $\CC$ consisting of column vectors whose coordinates are indexed by $X$ and whose entries are in $\CC$. We observe that $\MX$ acts on $V$ by left multiplication. We call $V$ the \emph{standard module}. We endow $V$ with the Hermitian inner 
product $\langle \, , \, \rangle$ that satisfies $\langle u,v \rangle = u^{\top}\overline{v}$ for 
$u,v \in V$, where $\top$ denotes the transposition and $\overline{\phantom{v}}$ denotes the complex conjugation. For $y \in X$, let $\widehat{y}$ denote the element of $V$ with $1$ in the ${y}$-coordinate and $0$ in all other coordinates. We observe that $\{\widehat{y}\;|\;y \in X\}$ is an orthonormal basis for $V$.

Let $A \in \MX$ denote the \emph{adjacency matrix} of $\Gamma$:
$$\left( A\right) _{yz}=
\begin{cases}
	\hspace{0.2cm} 1 \hspace{0.5cm} \text{if} & \partial(y,z)=1,   \\
	\hspace{0.2cm} 0 \hspace{0.5cm} \text{if} &  \partial(y,z) \neq 1
\end{cases} \qquad (y,z \in X).
$$
Let $M$ denote the subalgebra of $\MX$ generated by $A$. The algebra $M$ is called the \emph{adjacency algebra} of $\Gamma$. Observe that $M$ is commutative.

In order to recall the dual idempotents of $\Gamma$, we fix a vertex $x \in X$ for the rest of this section. Let $\varepsilon=\varepsilon(x)$ denote the eccentricity of $x$, that is, $\varepsilon = \max \{\partial(x,y) \mid y \in X\}$. For $ 0 \le i \le \varepsilon$, let $E_i^*=E_i^*(x)$ denote the diagonal matrix in $\MX$ with $(y,y)$-entry given by
\begin{eqnarray*}
	\label{den0}
	(\e_i)_{y y} = \left\{ \begin{array}{lll}
		1 & \hbox{if } \; \partial(x,y)=i, \\
		0 & \hbox{if } \; \partial(x,y) \ne i \end{array} \right. 
	\qquad (y \in X).
\end{eqnarray*}
We call $\e_i$ the \emph{$i$-th dual idempotent} of $\Gamma$ with respect to $x$ \cite[p.~378]{Terpart1}. We observe 
(ei)  $\sum_{i=0}^\varepsilon E_i^*=I$; 
(eii) $\overline{E_i^*} = E_i^*$ $(0 \le i \le \varepsilon)$; 
(eiii) $E_i^{*\top} = E_i^*$ $(0 \le i \le \varepsilon)$; 
(eiv) $E_i^*E_j^* = \delta_{ij}E_i^* $ $(0 \le i,j \le \varepsilon)$,
where $I$ denotes the identity matrix of $\MX$. It follows that $\{E_i^*\}_{i=0}^\varepsilon$ is a basis for a commutative subalgebra $M^*=M^*(x)$ of $\MX$. The algebra $M^*$ is called the \emph{dual adjacency algebra} of $\Gamma$ with respect to $x$ \cite[p.~378]{Terpart1}. Note that for $0 \le i \le \varepsilon$ we have
$\e_i V = {\rm Span} \{ \widehat{y} \mid y \in X, \partial(x,y)=i\}$, 
and  
\begin{equation}
	\label{vsub}
	V = E_0^*V + E_1^*V + \cdots + E_\varepsilon^*V \qquad \qquad {\rm (orthogonal\ direct\ sum}). \nonumber 
\end{equation}
The subspace $\e_i V$ is known as the \emph{$i$-th subconstituent of $\Gamma$ with respect to $x$}. 
For convenience, $\e_{i}$ is assumed to be the zero matrix of $\MX$ for $i<0$ and $i>\varepsilon$.

The \emph{Terwilliger algebra of $\Gamma$ with respect to $x$}, denoted by $T=T(x)$,  is the  subalgebra of $\MX$ generated by $M$ and $M^*$ \cite{Terpart1}. Observe that $T$ is generated by the adjacency matrix $A$ and the dual idempotents $E^*_i \, (0\leq i\leq \varepsilon)$, and so it is finite-dimensional. In addition, $T$ is closed under the conjugate-transpose map by construction, and is hence semi-simple. For a vector subspace $W \subseteq V$, we denote by $TW$ the subspace $\{B w \mid B \in T, w \in W\}$. Let us now recall the lowering, the flat, and the raising matrices of $T$. 

\begin{definition} \label{def2} 
	Let $\Gamma=(X,\R)$ denote a graph with diameter $D$. Pick $x \in X$, let $\varepsilon=\varepsilon(x)$ denote the eccentricity of $x$ and  $T=T(x)$ denote the Terwilliger algebra of $\Gamma$ with respect to $x$. Define $L=L(x)$, $F=F(x)$, and $R=R(x)$ in $\MX$ by
	\begin{eqnarray}\label{defLR}
		L=\sum_{i=1}^{\varepsilon}E^*_{i-1}AE^*_i, \hspace{1cm}
		F=\sum_{i=0}^{\varepsilon}E^*_{i}AE^*_i, \hspace{1cm}
		R=\sum_{i=0}^{\varepsilon-1}E^*_{i+1}AE^*_i. \nonumber 
	\end{eqnarray}
	We refer to $L$, $F$, and $R$ as the \emph{lowering}, the \emph{flat}, and the \emph{raising matrix with respect to $x$}, respectively.
\end{definition}
Note that, by definition, $L, F, R \in T$, $F=F^{\top}$, $R=L^{\top}$, and $A=L+F+R$.
Observe that for $y,z \in X$, the $(z,y)$-entry of $L$ equals $1$ if $\partial(z,y)=1$ and $\partial(x,z)= \partial(x,y)-1$ and $0$ otherwise. The $(z,y)$-entry of  $F$ is equal to $1$ if $\partial(z,y)=1$ and $\partial(x,z)= \partial(x,y)$ and $0$ otherwise. Similarly, the $(z,y)$-entry of $R$ equals $1$ if $\partial(z,y)=1$ and $\partial(x,z)= \partial(x,y)+1$ and $0$ otherwise. Consequently, for $v \in \e_i V \; (0 \le i \le \varepsilon)$ we have
\begin{equation}
	\label{eq:LRaction}
	L v \in \e_{i-1} V, \qquad  F v \in \e_{i} V, \qquad R v \in \e_{i+1} V.
\end{equation}
Observe also that $\Gamma$ is bipartite if and only if $F=0$.%

We now recall modules of $T$. By a \emph{$T$-module} we mean a subspace $W$ of $V$ such that $TW \subseteq W$. Let $W$ denote a $T$-module. Then, $W$ is said to be {\em irreducible} whenever $W$ is nonzero and $W$ contains no $T$-modules other than $0$ and $W$. Since the algebra $T$ is semi-simple, any $T$-module is an orthogonal direct sum of irreducible $T$-modules.

Let $W$ denote an irreducible $T$-module. We observe that $W$ is an orthogonal direct sum of the non-vanishing subspaces $E_i^*W$ for $0 \leq i \leq \varepsilon$. The \emph{endpoint} of $W$ is defined as $r:=r(W)=\min \{i \mid 0 \le i\le \varepsilon, \; \e_i W \ne 0 \}$, and the \emph{diameter} of $W$ as $d:=d(W)=\left|\{i \mid 0 \le i\le \varepsilon, \; \e_i W \ne 0 \} \right|-1 $. It turns out that $\e_iW \neq 0$ if and only if $r \leq i \leq r+d$ $(0 \leq i \leq \varepsilon)$. The module $W$ is said to be \emph{thin} whenever $\dim(E^*_iW)\leq1$ for $0 \leq i \leq \varepsilon$. We say that two $T$-modules $W$ and $W^{\prime}$ are \emph{$T$-isomorphic} (or simply \emph{isomorphic}) whenever there exists a vector space isomorphism $\sigma: W \rightarrow W^{\prime}$ such that $\left( \sigma B - B\sigma \right) W=0$ for all $B \in T$. Note that isomorphic irreducible $T$-modules have the same endpoint and the same diameter. It turns out that $T\widehat{x}=\{B\widehat{x} \; | \; B \in T\}$ is the unique irreducible $T$-module with endpoint $0$. We refer to $T\widehat{x}$ as the {\em trivial $T$-module}.


\section{Graphs that support a uniform structure}
\label{sec:uniform}

In this section, we discuss the uniform property of bipartite graphs. As already mentioned, the uniform property was first defined for graded partially ordered sets \cite{OldTer}. The definition was later extended to bipartite distance-regular graphs in \cite{MikTer}. We first extend this definition to an arbitrary bipartite graph. Let $\Gamma=(X, \mathcal{R})$ denote a bipartite graph and let $V$ denote the standard module of $\Ga$. Fix $x\in X$, and let $\varepsilon=\varepsilon(x)$ denote the eccentricity of $x$. Let $T=T(x)$, $L$, and $R$ denote the corresponding Terwilliger algebra, lowering, and raising matrix, respectively.
\begin{definition}\label{3diag}
	 A \emph{parameter matrix} $U=(e_{ij})_{1\leq i,j\leq \varepsilon}$ is defined to be a tridiagonal matrix with entries in $\CC$, satisfying the following properties:
	\begin{enumerate}[label=(\roman*)]
		\item $e_{ii}=1$ $(1\leq i\leq \varepsilon)$,
		\item $e_{i, i-1}\neq0$ for $2 \le i \le \varepsilon$ or $e_{i-1,  i}\neq0$ for $2\leq i\leq \varepsilon$, and
		\item  the principal submatrix $(e_{ij})_{s\leq i, \, j\leq t}$ is nonsingular for $1\leq s\leq t\leq \varepsilon$.
	\end{enumerate}
	For convenience we write $e^{-}_i:=e_{i, i-1}$ for $2\leq i\leq \varepsilon$ and $e^{+}_i:=e_{i, i+1}$ for $1\leq i\leq \varepsilon-1$. We also define $e^{-}_1:=0$ and $e^{+}_\varepsilon:=0$. 
\end{definition}

A \emph{uniform structure} of $\Gamma$ with respect to $x$ is a  pair $(U,f)$ where $f=\{f_i\}_{i=1}^\varepsilon$ is a vector in $\CC^\varepsilon$, such that
\begin{align}\label{uniformeq}
	e^{-}_iRL^2+LRL+ e^+_i L^2R=f_iL
\end{align}
 is satisfied on $E^*_iV$ for $1\leq i\leq \varepsilon$, where $\e_i\in T$ are the dual idempotents of $\Gamma$ with respect to $x$. In addition, a \emph{strongly uniform structure} of $\Gamma$ with respect to $x$ is  a uniform structure $(U,f)$ for which $e_i^-\neq0$ for $2\leq i\leq \varepsilon$ and $e_i^+\neq0$ for $1\leq i\leq \varepsilon-1$. If the vertex $x$ is clear from the context, we will simply use (\emph{strongly}) \emph{uniform structure} of $\Gamma$ instead of  (\emph{strongly}) \emph{uniform structure of $\Gamma$ with respect to $x$}. 
The following result of Terwilliger will be crucial in the rest of the paper. 

\begin{theorem}[{\cite[Theorem 2.5]{OldTer}}]\label{oldpaper}
Let $\Gamma=(X, \mathcal{R})$ denote a bipartite graph and fix $x\in X$. Let $T=T(x)$ denote the corresponding Terwilliger algebra. Assume that $\Gamma$ admits a uniform structure with respect to $x$. Then, the following (i), (ii) hold:
\begin{enumerate}[label=(\roman*),]
\item Every irreducible $T$-module is thin. 
\item Let $W$ denote an irreducible $T$-module with endpoint $r$ and diameter $d$. Then, the isomorphism class of $W$ is determined by $r$ and $d$. 
\end{enumerate}
\end{theorem}

Assume now that $\Ga$ is a non-bipartite graph, and fix a vertex $x$ of $\Ga$. In what follows we define what it means for $\Gamma$ to support a uniform structure with respect to $x$. To this end, we first define a certain graph $\Ga_f$.

\begin{definition}\label{def3.1}
	Let $\Gamma=(X,\mathcal{R})$ denote a non-bipartite graph. Fix $x \in X$ and define $\mathcal{R}_f = \mathcal{R} \setminus \{yz \mid \partial(x,y) = \partial(x,z)\}$. We define ${\Gamma_f}={\Gamma_f(x)}$ to be the graph with vertex set $X$ and edge set $\mathcal{R}_f$. We observe that $\Gamma_f$ is bipartite and connected. We say that $\Ga$ {\em supports a uniform structure with respect to} $x$, if $\Ga_f$ admits a uniform structure with respect to $x$. If the vertex $x$ is clear from the context, we will simply say that $\Ga$ \emph{supports a uniform structure} instead of \emph{supports a uniform structure with respect to $x$}. 
\end{definition}

With reference to Definition \ref{def3.1}, let $\varepsilon=\varepsilon(x)$ denote the eccentricity of $x$ and let $V$ denote the standard module for $\Ga$. Since the vertex set of $\Ga$ is equal to the vertex set of $\Ga_f$, observe that $V$ is also the standard module for $\Ga_f$. Let $T=T(x)$ be the Terwilliger algebra of $\Ga$. Recall that $T$ is generated by the adjacency matrix $A$ and the dual idempotents $\e_i$ ($0\leq i \leq \varepsilon$). Furthermore, we have $A=L+F+R$ where $L,F$, and $R$ are the corresponding lowering, flat, and raising matrices, respectively. Let $T_f=T_f(x)$ be the Terwilliger algebra of $\Ga_f$. As $\Ga_f$ if bipartite, the flat matrix of $\Ga_f$ with respect to $x$ is equal to the zero matrix. Moreover, the lowering and the raising matrices of $\Ga_f$ with respect to $x$ are equal to $L$ and $R$, respectively. It follows that for the adjacency matrix $A_f$ of $\Gamma_f$ we have $A_f=L+R$.  For $0 \le i \le \varepsilon$, note also that the $i$-th dual idempotent of $\Ga_f$ with respect to $x$ is equal to $\e_i$. Consequently, the algebra $T_f$ is generated by the matrices $L, R$, and $\e_i$ ($0\leq i \leq \varepsilon$).

\begin{lemma}
	With the above notation, let $W$ denote a $T$-module. Then,  the following (i), (ii) hold.

	\begin{enumerate}[label=(\roman*)]
		\item $W$ is a $T_f$-module.
		\item  If $W$ is a thin irreducible $T$-module, then $W$ is a thin irreducible $T_f$-module.
	\end{enumerate}
\end{lemma}
\begin{proof}
	$(i)$ This is clear since $T_f$ is generated by the matrices $L, R$, and $\e_i$ ($0\leq i \leq \varepsilon$).
  
  \noindent
	$(ii)$ Let $W'\subseteq W$ denote a non-zero irreducible $T_f$-module. Let $r'$ and $d'$ denote the endpoint and the diameter of $W'$, respectively. Observe that by \eqref{eq:LRaction} we have 
	$$ L\e_{r'}W'=0 \,\,\, \mbox{and} \,\,\,  R\e_{r'+d'}W'=0.$$
 Since $W$ is thin  we have  $\e_iW'=\e_iW$ for $r'\leq i \leq r'+d'$.	It follows by  \eqref{eq:LRaction} that $W'$ is also a $T$-module, and so $W'=W$.
\end{proof}

\begin{proposition}\label{ortho}
 With the above notation, let 
 $$ V=V_1+ V_2+\cdots +V_{\ell}\qquad(\mbox{orthogonal direct sum}),$$
  where $V_j  \, (1\leq j \leq \ell)$ are $T$-modules. Fix $0\leq i \leq \varepsilon$. Then, \eqref{uniformeq} holds on $\e_iV$ if and only if \eqref{uniformeq} holds on $\e_iV_j$ for every $1\leq j \leq \ell$. In particular, \eqref{uniformeq} holds on $\e_iV$ if and only if \eqref{uniformeq} holds on $\e_iW$ for every irreducible $T$-module $W$.
\end{proposition}
\begin{proof}
	If \eqref{uniformeq} holds on $\e_iV$, then it is clear that \eqref{uniformeq} holds on $\e_iV_j$ for every $1\leq j \leq \ell$. To prove the other direction, pick $v\in \e_iV$. Let 
	$$v=v_1+v_2+\cdots+v_\ell,$$
	where $v_j\in \e_iV_j$ for $1\leq j\leq \ell$. As \eqref{uniformeq} holds on $\e_iV_j$ for every $1\leq j \leq \ell$, we have that $e^{-}_iRL^2v+LRLv+ e^+_iL^2Rv =f_iLv$, and so \eqref{uniformeq} holds on $\e_iV$. The second part of the proposition now follows from the fact that $V$ could always be written as an orthogonal direct sum of irreducible $T$-modules.
\end{proof}

Recall that  if $\Gamma_f$ admits a uniform structure with respect to $x$, then an isomorphism class of an irreducible $T_f$-module is determined by the endpoint and diameter of this module. The following result will be useful in the remaining sections of this paper. 

\begin{proposition}\label{tilde}
With the above notation, let $W$ and $W'$ denote thin irreducible $T$-modules with endpoints $r$ and $r'$ and diameters $d$ and $d'$, respectively. Let $\{w_i\}_{i=0}^d$ be a basis for $W$ with $w_i \in \e_{r+i}W$. Similarly, let $\{w'_i\}_{i=0}^{d'}$ be a basis for $W'$ with $w'_i \in \e_{r'+i}W'$. Let $\beta_i \, (1\leq i \leq d)$ and $\gamma_i \, (0\leq i \leq d-1)$ be scalars such that $Lw_i=\beta_i w_{i-1}$ and $Rw_i=\gamma_i w_{i+1}$. Similarly, let $\beta'_i \, (1\leq i \leq d')$ and $\gamma'_i \, (0\leq i \leq d'-1)$ be scalars such that $Lw'_i=\beta'_i w'_{i-1}$ and $Rw'_i=\gamma'_i w'_{i+1}$. Then, $W$ and $W'$ are isomorphic as $T_f$-modules if and only if $r=r'$, $d=d'$, and $\beta_{i+1}/\beta'_{i+1}=\gamma'_i/\gamma_i$ for $0 \leq i \leq d-1$.
\end{proposition}

\begin{proof}
Assume first that $W$ and $W'$ are isomorphic as $T_f$-modules. Let $\sigma: W\rightarrow W'$ be a $T_f$-isomorphism. Since  $\left( \sigma \e_i - \e_i\sigma \right) W=0$, we have that $\sigma(w_i) \in \e_{r+i} W'$ for $0 \le i \le d$. Consequently, $r=r'$, $d=d'$, and that there exists non-zero scalars $\lambda_i$ $(0\leq i\leq d)$ such that $\sigma(w_i)=\lambda_i w'_i$.  As  $\left( \sigma L - L\sigma \right) W=0$, we get that $\beta_{i+1}\lambda_i=\lambda_{i+1}\beta'_{i+1}$ $(0\leq i\leq d-1)$. Similarly,  as  $\left( \sigma R - R\sigma \right) W=0$, we get that $\gamma'_{i}\lambda_i=\lambda_{i+1}\gamma_{i}$ $(0\leq i\leq d-1)$. Note that the scalars $\beta_i$ and $\gamma_i$ are non-zero, otherwise $W$ would contain a proper non-trivial submodule, which contradicts the irreducibility of $W$. The same holds for $\beta'_i$ and $\gamma'_i$. It follows that $\beta_{i+1}/\beta'_{i+1}=\gamma'_i/\gamma_i$ for $0 \leq i \leq d-1$. 

Next assume that $r=r'$, $d=d'$ and $\beta_{i+1}/\beta'_{i+1}=\gamma'_i/\gamma_i$ for $0 \leq i \leq d-1$. Let us define $\lambda_0=1$ and 
$$\lambda_{i+1}=\frac{\beta_{i+1}}{\beta'_{i+1}}\lambda_i=\frac{\gamma'_i}{\gamma_i}\lambda_i \qquad (0\leq i\leq d-1).$$ 
Furthermore, for $0 \le i \le d$ define $\sigma(w_i)=\lambda_iw'_i$, and extend $\sigma$ by linearity to $\sigma: W\rightarrow W'$. It is now easy to see that $\sigma$ is a vector-space isomorphism, and that $\left( \sigma B - B\sigma \right) W=0$ for $B\in \{L, R, \e_0,\ldots,\e_\varepsilon\}$. This yields that $\sigma$ is a $T_f$-isomorphism, which completes the proof.
\end{proof}


\section{Distance-regular graphs and $Q$-polynomial property}
\label{sec:drg}

In this section, we review some definitions and basic concepts regarding distance-regular graphs. Let $\Gamma=(X, \R)$ denote a graph with diameter $D$. Recall that for $x,y \in X$ the \emph{distance} between $x$ and $y$, denoted by $\partial(x,y)$, is the length of a shortest $xy$-path. 
For an integer $i$, we define $\Gamma_i(x)$ by 
$$\Gamma_i(x)=\left\lbrace y \in X \mid \partial(x, y)=i\right\rbrace. 
$$
In particular,  $\Gamma(x)=\Gamma_1(x)$ is the set of neighbors of $x$.

For an integer $k \geq 0$, we say that $\Gamma$ is \emph{regular} with valency $k$ whenever $|\Gamma(x)| = k$ for all $x\in X$. We say  that $\Gamma$ is \emph{distance-regular} whenever, for all integers $0 \leq h, i, j \leq D$ and all $x,y \in X$ with $\partial(x, y) = h$, the number $p^h_{ij}:= |\Gamma_{i}(x) \cap \Gamma_{j} (y)|$ is independent of the choice of $x,y$. The constants $p^h_{ij}$ are known as the \emph{intersection numbers} of $\Gamma$. For convenience, set $c_i := p^i _{1 \, i-1}  \; (1 \leq i \leq D)$, $a_i := p^i_{ 1i} \; (0 \leq i \leq D)$, $b_i := p^i_{ 1 \, i+1} \; (0 \leq i \leq D - 1)$,  $k_i := p^0_{ii} \; (0 \leq i \leq D)$, and $c_0 := 0, b_D := 0$.  Note that in a distance-regular graph, the eccentricity of every vertex is equal to $D$. 
For the rest of this section assume that $\Gamma$ is distance-regular. By the triangle inequality, for $0 \leq h, i, j \leq D$ we have $p^h_{ij}= 0$ (resp., $p^h_{ij}\neq 0$) whenever one of $h, i, j$ is greater than (resp., equal to) the sum of the other two. In particular, $c_i \neq 0 $ for $ 1 \leq i \leq D$ and $b_i \neq 0$ for $0 \leq i \leq D - 1$. Observe that $\Gamma$ is regular with valency $k = b_0 = k_1$ and $c_i +a_i +b_i = k $ for $ 0 \leq i \leq D$. Moreover, $\Ga$ is bipartite if and only if $a_i=0$ for $0 \le i \le D$.

We now recall the near polygons. A distance-regular graph $\Ga$ is called a {\em near polygon} whenever $a_i=a_1 c_i$ for $1 \le i \le D-1$ and $\Ga$ does not contain $K_{1,1,2}$ as an induced subgraph \cite{No}. Here $K_{1,1,2}$ denotes the complete multipartite graph with three parts of order $1$, $1$, and $2$, respectively. 

We now recall what it means for
$\Ga$ to have classical parameters and negative type.
The graph $\Ga$ is said to have {\it classical parameters} $(D,q,\alpha,\beta)$ 
whenever the intersection numbers of $\Ga$ satisfy
$$
  c_i = {i \brack 1} \Big( 1+\alpha {i-1 \brack 1} \Big) \qquad \qquad (1 \le i \le D), 
$$
$$
  b_i = \Big( {D \brack 1} - {i \brack 1} \Big) 
        \Big( \beta-\alpha {i \brack 1} \Big) \qquad \qquad (0 \le i \le D-1), 
$$
where
$$
  {j \brack 1} := 1+q+q^2+ \cdots + q^{j-1}. 
$$
In this case $q$ is an integer and $q \not \in \{0,-1\}$, see \cite[Proposition 6.2.1]{BCN}. We say that $\Ga$ is of {\it negative type} whenever $\Ga$ has classical parameters $(D,q,\alpha,\beta)$ such that $q < -1$. Distance-regular graphs with classical parameters with $q=1$ are classified (see \cite[Theorem 6.1.1]{BCN}).

\medskip \noindent 
Next we recall the distance matrices of $\Ga$. For $0 \le i \le D$ let $A_i$ denote the matrix in  $\MX$ with $(y,z)$-entry
\begin{equation}
\label{dm1}
  (A_i)_{y z} = \left\{ \begin{array}{ll}
                 1 & \hbox{if } \; \partial(y,z)=i, \\
                 0 & \hbox{if } \; \partial(y,z) \ne i \end{array} \right. \qquad (y,z \in X).
\end{equation}
We call $A_i$ the $i$th {\it distance matrix} of $\Ga$. Note that $A_1=A$ is the
adjacency matrix of $\Ga$. We observe
(ai)   $A_0 = I$;
(aii)  $\sum_{i=0}^D A_i=J$;
(aiii)  $A_i^\top = A_i  \;(0 \le i \le D)$;
(aiv)   $A_iA_j = \sum_{h=0}^D p_{ij}^h A_h \;(0 \le i,j \le D)$,
where $I$ (resp. $J$) denotes the identity matrix (resp. all 1's matrix) in  $\MX$. It turns out that $\{A_i\}_{i=0}^D$ is a basis for the adjacency algebra $M$ of $\Ga$. By 
\cite[p.~45]{BCN} $M$ has a second basis $\{E_i\}_{i=0}^D$ such that 
(ei)   $E_0 = |X|^{-1}J$;
(eii)  $\sum_{i=0}^D E_i=I$;
(eiii)  $E_i^\top =E_i  \;(0 \le i \le D)$;
(eiv)   $E_iE_j =\delta_{ij}E_i  \;(0 \le i,j \le D)$.
We call $\{E_i\}_{i=0}^D$  the {\it primitive idempotents} of $\Gamma$.  
The primitive idempotent $E_0$ is said to be {\em trivial}. Since $\{E_i\}_{i=0}^D$ form a basis for $M$,
there exist scalars $\{\theta_i\}_{i=0}^D$ such that $A = \sum_{i=0}^D \theta_i E_i$. Combining this with (eiv) we find
$$
  A E_i = E_i A = \theta_i E_i \qquad \qquad (0 \le i \le D).
$$ 
We call $\theta_i$ the {\em eigenvalue} of $\Ga$ associated with $E_i$. 
It turns out that  $\{\theta_i\}_{i=0}^D$ are real and mutually distinct since $A$ generates $M$. By (ei) we have $\theta_0 = k$. 
By (eii)--(eiv), 
\begin{equation}
\label{de}
  V = E_0V + E_1V + \cdots + E_DV \qquad \hbox{(orthogonal direct sum)}.
\end{equation}
For $0 \le i \le D$ the space $E_iV$ is the eigenspace of $A$ associated with $\theta_i$.
Let $m_i$ denote the rank of $E_i$, and note that $m_i$ is the dimension of $E_iV$.
We call $m_i$ the {\em multiplicity} of $\theta_i$. 

\medskip \noindent
We recall the Krein parameters of $\Ga$.
Let $\circ $ denote the entrywise product in $\MX$. Observe that
$A_i\circ A_j= \delta_{ij}A_i$ for $0 \leq i,j\leq D$, so $M$ is closed under $\circ$. Thus, there 
exist scalars $q^h_{ij}$  $(0 \leq h,i,j\leq D)$ such that
$$
  E_i \circ E_j = |X|^{-1}\sum_{h=0}^D q^h_{ij}E_h \qquad \qquad (0 \leq i,j\leq D).
$$
The parameters $q^h_{ij}$ are called the {\it Krein parameters of} $\Ga$.
By \cite[Proposition 4.1.5]{BCN} these parameters are real and nonnegative.
The given ordering $\{E_i\}_{i=0}^D$ of the primitive idempotents
is said to be $Q$-{\em polynomial} if for $0 \le h,i,j \le D$ 
the Krein parameter $q^h_{ij}=0$ (resp. $q^h_{ij} \ne 0$) whenever one of $h,i,j$ is greater than (resp. equal to) 
the sum of the other two. Let $E$ denote a nontrivial primitive idempotent of $\Ga$ and let $\theta$ denote the 
corresponding eigenvalue. We say that $\Ga$ is {\em Q-polynomial with respect to E} (or $\theta$) whenever there exists a 
$Q$-polynomial ordering $\{E_i\}_{i=0}^D$ of the primitive idempotents of $\Ga$ such that $E_1 = E$. 

The ordering $\{E_i\}_{i=0}^D$ of the primitive idempotents of $\Ga$ (and the corresponding ordering of eigenvalues $\{\theta_i\}_{i=0}^D$ of $\Ga$) is said to be \emph{original} if $\theta_0 > \theta_1 > \cdots > \theta_D$. By \cite[Corollary 8.4.2]{BCN}, all graphs with classical parameters are $Q$-polynomial. Moreover, if $\Ga$ has classical parameters with $q \ge 1$, then the original ordering of the primitive idempotents of $\Ga$ is $Q$-polynomial. 

Assume for a moment that $\{E_i\}_{i=0}^D$ is a $Q$-polynomial ordering of the primitive idempotents of $\Ga$. Pick $x \in X$, and let $T=T(x)$ denote the corresponding Terwilliger algebra. Let $W$ denote a thin irreducible $T$-module with endpoint $r$ and diameter $d$. We observe that $W$ is an orthogonal direct sum of the non-vanishing subspaces $E_iW$ for $0 \leq i \leq D$. The \emph{dual endpoint} of $W$ is defined as $t:=t(W)=\min \{i \mid 0 \le i\le D, \; E_i W \ne 0 \}$, and the \emph{dual diameter} of $W$ as $d^*:=d^*(W)=\left|\{i \mid 0 \le i\le D, \; E_i W \ne 0 \} \right|-1$. From \cite[Lemma 3.1, Lemma 5.1]{CMT} we deduce that $d=d^*$. Pick a nonzero $v \in E_t W$. Then, $\{\e_i v\}_{i=r}^{r+d}$ is a basis for $W$, see \cite[Theorem 8.1(i)]{cerzo}. We call this basis a \emph{standard basis} for $W$. In what follows we will be using the following notation.

\begin{notation}\label{not3.2}
	Let $\Ga=(X,\mathcal{R})$ denote a distance-regular non-bipartite graph with diameter $D \ge 3$, intersection numbers $b_i \, (0 \le i \le D-1)$, $c_i \, (1 \le i \le D)$,  and eigenvalues $\theta_0>\theta_1>\ldots>\theta_D$. Let $E_0, E_1, \ldots, E_D $ denote the corresponding primitive idempotents of $\Gamma$ and $V$ denote the standard module of $\Ga$. Fix $x\in X$. Let $T=T(x)$ be the Terwilliger algebra of $\Ga$ and $\e_i \, (0 \le i \le D)$ be the dual idempotents of $\Ga$ with respect to $x$. Let $L$, $F$, and $R$ denote the corresponding lowering, flat, and raising matrix, respectively. Let ${\Gamma_f}={\Gamma_f(x)}$ be as defined in Definition~\ref{def3.1}. Let $T_f=T_f(x)$ be the Terwilliger algebra of $\Ga_f$. Recall  that $T_f$ is generated by the matrices $L, R$, and $\e_i$ ($0\leq i \leq D$).
\end{notation}


\section{Irreducible $T$-modules of distance-regular graphs} \label{endp1}

With reference to Notation \ref{not3.2}, in this section, we summarize certain results of Go and Terwilliger \cite{GoTer, Ter} about irreducible $T$-modules with endpoint $1$. 

\begin{definition} \label{dual}
With reference to Notation \ref{not3.2}, let $W$ denote a thin irreducible $T$-module with endpoint $r$. Pick a nonzero $w \in \e_r W$. Since $\dim(\e_rW)=1$, we have that $Fw = \eta w$ for some scalar $\eta$. Observe that $\eta$ does not depend on the choice of $w$. We refer  to $\eta$ as the \emph{local eigenvalue of} $W$. 
\end{definition}

With reference to Notation \ref{not3.2}, to further describe local eigenvalues of irreducible $T$-modules with endpoint $1$, we introduce the following notation: for $z\in \mathbb{R}$, we define 
$$ 
\tilde{z}=\begin{cases}
	-1-\frac{b_1}{1+z} &\hbox{ if $z\neq-1$},\\
	\infty &\hbox{ if $z=-1$}.
\end{cases}
$$

It follows from \cite[Lemma 2.6]{JKT} that $\theta_1,\theta_D$ are not equal to $ -1$, that $\tilde{\theta_1}<-1$ and that $\tilde{\theta_D}\ge 0$. Let $W$ denote a thin irreducible $T$-module with endpoint 1 and local eigenvalue $\eta$.  By \cite[Theorem 1]{Ter1}, we have that $\tilde{\theta_1}\le \eta \le \tilde{\theta_D}$.

\begin{proposition}[{\cite[Theorem 10.3]{GoTer}, \cite[Theorem 11.4]{Ter}}]\label{prop:diam}
	With reference to Notation \ref{not3.2}, let $W$ denote a thin irreducible $T$-module with endpoint $1$, diameter $d$, and local eigenvalue $\eta$. Then, the following (i), (ii) hold. 
	\begin{enumerate}[label=(\roman*)]
		\item If $\eta \in \{\tilde{\theta_1}, \tilde{\theta_D}\}$, then $d=D-2$. 
		\item If $\tilde{\theta_1}<\eta< \tilde{\theta_D}$, then $d=D-1$. 
	\end{enumerate}
 \end{proposition}

\begin{proposition} \label{prop:dualendpt}
	With reference to Notation \ref{not3.2}, assume that $\Ga$ is $Q$-polynomial with respect to a certain ordering of its primitive idempotents. Let $W$ denote a thin irreducible $T$-module with endpoint $1$, diameter $d$, dual endpoint $t$, and local eigenvalue $\eta$. Then, the following (i), (ii) hold. 
	\begin{enumerate}[label=(\roman*)]
		\item If $\eta \in \{\tilde{\theta_1}, \tilde{\theta_D}\}$, then $t \in \{1,2\}$.
		\item If $\tilde{\theta_1}<\eta< \tilde{\theta_D}$, then $t=1$.
	\end{enumerate}
\end{proposition}
\begin{proof}
	First, assume that $\eta \in \{\tilde{\theta_1},\tilde{\theta_D}\}$, and so $d=D-2$ by Proposition \ref{prop:diam}(i). By \cite[Lemma 7.1(ii)]{Cau}, the dual endpoint $t\geq1$. Furthermore, from \cite[eq. (39)]{Cau}, we derive that $t\le2$.
	
	Now, suppose that $\tilde{\theta_1} <\eta<\tilde{\theta_D}$, and so $d=D-1$ by Proposition \ref{prop:diam}(ii). Similarly, as above, we derive from  \cite[Lemma 7.1(ii), eq. (39)]{Cau} that $t=1$ in this case.
\end{proof}

We now recall tight distance-regular graphs. With reference to Notation \ref{not3.2}, by \cite[Theorem 6.2]{JKT}, we have that 
\begin{equation}\label{fb}
	\left(\theta_1+\frac{b_0}{a_1+1}\right) \left(\theta_D+\frac{b_0}{a_1+1}\right) \geq -\frac{b_0 a_1 b_1}{(a_1+1)^2}.
\end{equation}
 We say that $\Gamma$ is \emph{tight} whenever the equality holds in \eqref{fb}.

\begin{theorem}[{\cite[Theorem 13.6, Theorem 13.7]{GoTer}}]\label{tight}
	With reference to Notation \ref{not3.2}, the following (i)--(iii) are equivalent:
	\begin{enumerate}[label=(\roman*)]
		\item $\Gamma$ is tight,
		\item every irreducible $T$-module with endpoint $1$ is thin with local eigenvalue $\tilde{\theta_1}$ or $\tilde{\theta_D}$,
		\item $a_D=0$ and every irreducible $T$-module with endpoint $1$ is thin.
	\end{enumerate}
\end{theorem}

\begin{theorem}\label{th2.3}
	With reference to Notation \ref{not3.2}, assume that $\Gamma$ is tight. Then $\Gamma$ has, up to isomorphism, exactly two irreducible $T$-modules with endpoint $1$, which are both thin. One of these modules has local eigenvalue  $\tilde{\theta_1}$, while the other one has  local eigenvalue $\tilde{\theta_D}$. 
\end{theorem}
\begin{proof}
	Let $W$ and $W'$ denote irreducible $T$-modules with endpoint 1. Recall that by Theorem \ref{tight} the modules $W$ and $W'$ are both thin with local eigenvalues $\tilde{\theta_1}$ or $\tilde{\theta_D}$. By \cite[Theorem 11.1]{GoTer},  the modules $W$ and $W'$ are isomorphic if and only if they have the same local eigenvalue. It follows that $\Ga$ has, up to isomorphism, at most two irreducible modules with endpoint $1$. Assume for a moment that $\Gamma$ has, up to isomorphism, a unique irreducible $T$-module with endpoint 1. Since $\Gamma$ is not bipartite, it follows from \cite[Theorem 1.3]{CN} that $\Gamma$ is almost bipartite ($\Gamma$ is \emph{almost bipartite} if $a_i=0$ for $0\leq i\leq D-1$, and $a_D\neq0$), contradicting Theorem \ref{tight}. This completes the proof.
\end{proof}


\section{Distance-regular graphs of negative type with $a_1\neq0$} \label{sec3}

We are now ready to start our investigation about which graphs with classical parameters with $q \le 1$ support a uniform structure. With reference to Notation \ref{not3.2},  we first consider distance-regular graphs of negative type. We split our analysis into three cases: $\Gamma$ has intersection number $a_1\neq0$ and is not a near polygon, $\Gamma$ has intersection number $a_1=0$, and $\Gamma$ is a near polygon. In the next three sections, we study each of these cases separately. In this section,  we therefore assume that  $\Gamma$ is a distance-regular graph of negative type with $a_1\neq0$, which is not a near polygon. We prove that $\Gamma$ does not support a uniform structure with respect to $x$. To show this, we use the fact that in this case $\Gamma$ has, up to isomorphism,  a unique non-thin irreducible $T$-module with endpoint 1 \cite[Proposition 14.2]{miklavivc2009terwilliger}, which  we now describe.

\begin{proposition}[{\cite[Theorem 12.2, Theorem 14.4]{miklavivc2009terwilliger}}]\label{non-thin}
	With reference to Notation \ref{not3.2}, assume that $\Gamma$ is of negative type with $a_1\neq0$ and it is  not a near polygon. Then,  the following (i), (ii) hold.
	\begin{enumerate}[label=(\roman*)]
		\item There exists, up to isomorphism, a unique irreducible non-thin $T$-module with endpoint $1$.
		\item Let $W$ denote a non-thin irreducible $T$-module with endpoint $1$. Pick a non-zero $w\in E^*_1W$. Then, the following vectors form a basis for $W$:
		\begin{align}\label{basis:a1ne0}
			E^*_iA_{i-1}w    \quad (1\leq i \leq D), \hspace*{0.5cm}	E^*_iA_{i+1}w    \quad (2\leq i \leq D-1).
		\end{align}
	\end{enumerate}
\end{proposition}

Next, we give the action of the lowering and raising matrices on the basis vectors from \eqref{basis:a1ne0}.

\begin{lemma}[{\cite[Lemma 13.1]{miklavivc2009terwilliger}}]
	\label{nimeact:lem1}
	With reference to Notation \ref{not3.2}, assume that $\Gamma$ is of negative type with $a_1\neq0$ and it is  not a near polygon. Let $W$ denote a non-thin irreducible $T$-module with endpoint $1$. Then, the following (i)--(v) hold for a nonzero $w \in \e_1 W$:
	
		\begin{enumerate}[label=(\roman*)]
		\item $L w = 0$,
		\item $L \e_2 A w = (b_0 - c_2(a_1+1)) w$,
		\item $L \e_i A_{i-1} w = b_{i-1} \e_{i-1} A_{i-2} w + (c_i - c_{i-1}) \e_{i-1} A_i w$ \: \mbox{for} \,\,$3 \le i \le D$,
		\item $L \e_2 A_3 w = - b_2 (a_1+1) w$,
		\item $L \e_i A_{i+1} w = b_i \e_{i-1} A_i w$\: \mbox{for} \,\,$3 \le i \le D-1.$
	\end{enumerate}
\end{lemma}

\begin{lemma}[{\cite[Lemma 13.3]{miklavivc2009terwilliger}}]
	\label{nimeact:lem3}
	With reference to Notation \ref{not3.2}, assume that $\Gamma$ is of negative type with parameters $(D,q,\alpha, \beta)$. Assume that $a_1\neq0$ and $\Ga$ is not a near polygon. Let $W$ denote a non-thin irreducible $T$-module with endpoint $1$. Then, the following (i)--(iv) hold for a nonzero $w \in \e_1 W$:
		\begin{enumerate}[label=(\roman*)]
		\item $R \e_i A_{i-1} w = c_i \e_{i+1} A_i w$\: \mbox{for} \,\,$1 \le i \le D-1,$
		\item $R \e_D A_{D-1} w = 0,$
		\item for $2 \le i \le D-2$,
	               \begin{eqnarray*}
						R \e_i A_{i+1} w &=& (a_1+1)(c_{i+1} (q^{i-1}-1) (q^{i+1}-1)^{-1} - c_i) \e_{i+1} A_i w +
                                                                          \\&&c_{i+1} (q^{i-1}-1) (q^{i+1}-1)^{-1} \e_{i+1} A_{i+2} w,
					\end{eqnarray*} 
		\item $R \e_{D-1} A_D w = (a_1 + 1) (c_D (q^{D-2}-1) (q^D-1)^{-1} - c_{D-1}) \e_D A_{D-1} w.$
	\end{enumerate}
\end{lemma}

\noindent
We now prove the main theorem of this section.
\begin{theorem}\label{nonzero}
	With reference to Notation \ref{not3.2}, assume that $\Gamma$ is of negative type with $a_1\neq0$ and it is  not a near polygon. Then, $\Gamma$ does not support a uniform structure with respect to $x$.
\end{theorem}
\begin{proof}
	Let $W$ denote a non-thin irreducible  $T$-module with endpoint 1 and pick a non-zero $w\in E^*_1W$. Recall that $W$ is also a $T_f$-module. Let $W'\subseteq W$ be an irreducible $T_f$-module which contains $w$ (note that such a $T_f$-module exists as $\dim(\e_1W)=1$). We show that $W'$ is non-thin. To do this observe that by Lemma~\ref{nimeact:lem1} and Lemma~\ref{nimeact:lem3}  we have that  
	$$
	  Rw=E^*_2Aw  \quad \hbox{and} \quad LR^2w=c_2b_2E^*_2Aw+ c_2(c_3-c_2)E^*_2A_3w .
	  $$
	Moreover, both of these two vectors are contained in $E^*_2W'$. To show that $W'$ is non-thin, it is enough to prove that these two vectors are linearly independent. Assume that $\lambda Rw+ \mu LR^2w=0$ for some scalars $\lambda$ and $\mu$. This implies that
 $$(\lambda+\mu c_2b_2)E^*_2Aw+\mu c_2(c_3-c_2)E^*_2A_3w=0,$$
 and since $E^*_2Aw$ and $E^*_2A_3w$ are linearly independent by Proposition \ref{non-thin}, we have $\lambda+\mu c_2b_2=0$ and $\mu c_2(c_3-c_2)=0$. By \cite[Proposition 6.1.2]{BCN}, we have $c_3>c_2$, which implies $\lambda=\mu=0$. This shows that $W'$ is non-thin, and so $\Gamma$ does not support a uniform structure by Theorem~\ref{oldpaper}(i).  
\end{proof}


\section{Distance-regular graphs of negative type with $a_1=0$} \label{sec:a1=0}

With reference to Notation \ref{not3.2}, we continue our analysis of distance-regular graphs of negative type. In this section, we assume that  $\Gamma$ is of negative type with $a_1 =0$. We show that $\Gamma$ does not support a uniform structure. Since our approach is similar to the one in Section~\ref{sec3}, we only provide a sketch of the proof.

\begin{lemma}
	With reference to Notation \ref{not3.2}, assume that $\Gamma$ is of negative type with classical parameters $(D,q,\alpha,\beta)$ and $a_1=0$. Then, $a_2\neq0$.
\end{lemma}

\begin{proof}
	Assume that $a_2=0$. Solving $a_1=0$ and $a_2=0$ for $\alpha$ and $\beta$, we get that $\alpha=0$ and $\beta=1$. But then $\Gamma$ is bipartite by \cite[Proposition 6.3.1 (i)]{BCN}, a contradiction.
\end{proof}

By \cite[Theorem~10.1]{miklavivc2009q}, $\Gamma$ has, up to isomorphism, a unique irreducible $T$-module $W$ with endpoint $1$. Pick a non-zero $w \in \e_1W$. Then, by  \cite[Theorem~8.5]{miklavivc2009q} the following is a basis for $W$:
\begin{align}\label{basis}
	\e_iA_{i-1}w,\qquad 	\e_{i+1}A_{i+1}w\qquad (1\leq i\leq D-1).
\end{align} 

\begin{theorem}\label{zero}
	With reference to Notation \ref{not3.2}, assume that $\Gamma$ is of negative type with $a_1=0$. Then, $\Gamma$ does not support a uniform structure with respect to $x$.
\end{theorem}

\begin{proof}
Let $W$ denote an irreducible $T$-module with endpoint $1$ and pick a non-zero $w\in \e_1W$. Let $W'$ denote an irreducible $T_f$-module containing $w$. Then, using  \cite[Lemma~9.1, Lemma~9.3, Theorem~11.2]{miklavivc2009q},  we find that $Rw=\e_2Aw$ and
\begin{align*}
LR^2w=	\begin{cases}
		c_2(b_2+c_2-c_3)\e_2Aw+c_2(c_2-c_3)\e_2A_2w & \hbox{if} \;  D\geq4,\\
		-c_2(a_2-a_3)\e_2Aw-c_2(b_2+a_2-a_3)\e_2A_2w & \hbox{if} \; D=3.
	\end{cases}
\end{align*}

Assume that $\lambda Rw+\mu LR^2w=0$ for some scalars $\lambda$ and $\mu$. If $D\geq4$, then $\lambda Rw+\mu LR^2w=0$ is equivalent to 
\begin{align*}
	(\lambda +\mu c_2 (b_2+c_2-c_3))\e_2Aw+\mu c_2 (c_2-c_3)\e_2A_2w=0.
\end{align*}
Since $\e_2Aw$ and $\e_2A_2w$ are linearly independent by \eqref{basis}  and $c_3>c_2$ by \cite[Proposition~6.1.2]{BCN}, we have that $\lambda=\mu=0$, which implies that $ Rw$ and $ LR^2w$ are linearly independent. Now if $D=3$, then $\lambda Rw+\mu LR^2w=0$ is equivalent to 
\begin{align*}
	(\lambda -\mu c_2 (a_2-a_3))\e_2Aw-\mu c_2(b_2+a_2-a_3)\e_2A_2w=0.
\end{align*}

If $b_2+a_2-a_3=0$, then we have that $c_2=c_3$ contradicting \cite[Proposition~6.1.2]{BCN}. Therefore, $\lambda=\mu=0$, which implies that $ Rw$ and $ LR^2w$ are linearly independent. This shows that $W'$ is non-thin, and so $\Ga$ does not support a uniform structure by Theorem~\ref{oldpaper}(i).
\end{proof}

We summarize the results of the previous two sections in the following theorem.

\begin{theorem} \label{thm:zero}
		With reference to Notation \ref{not3.2}, assume that $\Gamma$ is of negative type and it is not a near polygon. Then, $\Gamma$ does not support a uniform structure with respect to $x$.
\end{theorem}
\begin{proof}
	If $a_1\neq0$, then this follows from Theorem~\ref{nonzero}. If $a_1=0$, then this follows from Theorem~\ref{zero}.
\end{proof}


\section{Regular near polygons of negative type}\label{sec:rnp}

With reference to Notation \ref{not3.2}, in this section we conclude our analysis of distance-regular graphs of negative type by studying regular near polygons (of negative type), that support a uniform structure. We first recall the following classification result by  Chih-wen Weng.

\begin{theorem}[{\cite{weng}}]\label{thm:negtype}
	With reference to Notation \ref{not3.2}, assume $\Ga$ has classical parameters $(D,q,\alpha,\beta)$ and $D \ge 4$. Suppose $q < -1$ and  the intersection numbers $a_1 \ne 0, c_2  >1$. Then, precisely one of the following (i)--(iii) holds.
	\begin{enumerate}[label=(\roman*)]
		\item $\Ga$ is the dual polar graph $^2 A_{2D-1}(-q)$ (see \cite[Section 9.4]{BCN}).
		\item $\Ga$  is the Hermitian forms graph $Her_{-q}(D)$ (see \cite[Section 9.5C]{BCN}).
		\item $\alpha=(q-1)/2$, $\beta=-(1+q^D)/2$, and $-q$ is a power of an odd prime.
	\end{enumerate}
\end{theorem}
We have the following corollary of the above result.

\begin{corollary}\label{cor:negtype}
	With reference to Notation \ref{not3.2}, assume $\Ga$ has classical parameters $(D,q,\alpha,\beta)$. Suppose that $\Ga$ is a regular near polygon with $q < -1$. Then, either $\Ga$ is the dual polar graph $^2 A_{2D-1}(-q)$ or $D=3$.
\end{corollary}
\begin{proof}
	We will assume that $D \ge 4$ and show that in this case $\Ga$ is the dual polar graph $^2 A_{2D-1}(-q)$. Since $\Ga$ is a regular near polygon, we have that $a_2 = a_1 c_2$. Solving $a_2 = a_1 c_2$ for $\beta$, we get that \begin{equation}\label{eq:beta}
		\beta=-\frac{\alpha q^D+q^2-\alpha q-q}{q-1}. 
	\end{equation}
Note also that if $\beta$ is as in \eqref{eq:beta}, then $a_i = a_1 c_i$ for every $1 \le i \le D$. If $a_1=0$, then it follows that $a_i=0$ for $1 \le i \le D$. This yields that $\Ga$ is bipartite, contradicting assumptions in Notation \ref{not3.2}. Therefore, $a_1 \ne 0$. Assume next that $c_2=1$. This yields $\alpha= -q/(q+1)$. Using \eqref{eq:beta}, we get $\beta=(q^{D+1}-q^3-q^2+q)/(q^2-1)$. By \cite[Corollary 5.4]{dBV}, no such graphs exist. This shows that $c_2 \ge 2$, and so $\Ga$ is one of the graphs from Theorem \ref{thm:negtype}. It follows from \eqref{eq:beta} and \cite[p. 194, Table 6.1]{BCN} that the Hermitian forms graph $Her_{-q}(D)$ is not a regular near polygon, while the dual polar graph $^2 A_{2D-1}(-q)$ is a regular near polygon (note that the dual polar graph $^2 A_{2D-1}(-q)$  corresponds to the graph $U(2D,-q)$ in \cite[p. 194, Table 6.1]{BCN}). Assume finally that $\alpha=(q-1)/2$, $\beta=-(1+q^D)/2$. It follows from \eqref{eq:beta} that $q=1$, a contradiction. This completes the proof.
\end{proof}
\begin{remark}
	Examples of regular near polygons of negative type with $D=3$ include the Triality graph $^3 D_{4,2}(-q)$ (see \cite[Section 10.7]{BCN}), the Witt graph $M_{24}$ (see \cite[Section 11.4A]{BCN}), and the extended ternary Golay code graph (see \cite[Section 11.3A]{BCN}. 
\end{remark}

\begin{theorem}\label{thm:dualpolar}
	With reference to Notation~\ref{not3.2}, let $\Ga$ denote  the dual polar graph $^2 A_{2D-1}(-q)$. Then, 
	\begin{equation}\label{eq:dualpolar}
		- \frac{q^4}{q^2+1}RL^2 + LRL - \frac{q^{-2}}{q^2+1}L^2R=(-q)^{2D-1}L
	\end{equation}
	is satisfied on $\e_iV$ for $1 \le i \le D$.
	In particular, $\Ga$ supports a strongly uniform structure with respect to $x$, where $e_i^{-}=- q^4/(q^2+1)\, (2\leq i\leq D)$, $e_i^{+}=-q^{-2}/(q^2+1)\, (1\leq i\leq D-1)$, and $f_i=(-q)^{2D-1}\,(1\leq i\leq D)$.  
\end{theorem}
\begin{proof}
	From \cite[Proposition 26.4(i)]{w:dual}, it follows that \eqref{eq:dualpolar} is satisfied on $\e_iV$ for $1 \le i \le D$.
	
Recall the scalars $e_i^-$, $e_i^+$, and let  $U$ be the corresponding parameter matrix from Definition~\ref{3diag}. Clearly $e_i^{-}\neq 0$ and $e_i^{+}\neq 0$. For $1\leq s\leq t\leq D$,  let $P=(e_{ij})_{s\leq i,j \leq t}$ denote the principal submatrix of $U$. A simple linear recurrence argument shows that 
	$$
	  \det(P)=\frac{q^{2(t-s+2)}-1}{(q^2-1)(q^2+1)^{t-s+1}},
	 $$
   implying that $P$ is non-singular. In virtue of Proposition~\ref{ortho}, $\Ga$ supports a strongly uniform structure with respect to $x$.   
\end{proof}

\noindent
We summarize the results of the previous three sections in the following theorem.
\begin{theorem}\label{thm:main_neg_type}
	With reference to Notation~\ref{not3.2}, assume that $\Ga$ is of negative type. Then, the following (i), (ii) hold.
	\begin{enumerate}[label=(\roman*)]
		\item If $D \ge 4$, then $\Ga$ supports a uniform structure with respect to $x$ if and only if $\Ga$ is the dual polar graph $^2 A_{2D-1}(-q)$. Moreover, $^2 A_{2D-1}(-q)$ supports a strongly uniform structure with respect to $x$.
		\item If $D=3$ and $\Ga$ supports a uniform structure with respect to $x$, then $\Ga$ is a regular near polygon.
	\end{enumerate}
\end{theorem}
\begin{proof}
	Immediately from Theorem \ref{nonzero}, Theorem \ref{thm:zero}, Corollary \ref{cor:negtype}, and Theorem \ref{thm:dualpolar}.
\end{proof}


\section{Graphs with classical parameters with $q=1$}\label{sec:q=1}

With reference to Notation \ref{not3.2}, in this section we classify non-bipartite distance-regular graphs with classical parameters with $q=1$ that support a uniform structure. It turns out that distance-regular graphs with classical parameters with $q=1$ are classified.
\begin{theorem}[{\cite[Theorem 6.1.1]{BCN}}]
	Let $\Gamma$ denote a distance-regular graph with classical parameters with $q=1$. Then, $\Gamma$ is one of the following graphs: 
\begin{enumerate}[label=(\roman*)]
		\item Johnson graph $J(n,D)$, $n\geq2D$ (see \cite[Section~9.1]{BCN}),
    	\item Gosset graph (see \cite[Section~3.11]{BCN}),
		\item Hamming graph $H(D,n)$ (see \cite[Section~9.2]{BCN}),
    	\item Halved cube  $\frac{1}{2}H(n,2)$ (see \cite[Section~9.2D]{BCN}),
    	\item Doob graph $D(n,m)$, $n\geq1$, $m\geq0$ (see \cite[Section~9.2B]{BCN}).
\end{enumerate}
\end{theorem}
In what follows we analyze each of the above cases separately. Recall also that all of the above graphs are $Q$-polynomial with respect to the original ordering of its primitive idempotents. We will be using this fact throughout this section. Since some of these graphs are tight, we start our analysis by considering tight graphs first.

\subsection{Tight graphs with classical parameters}
With reference to Notation~\ref{not3.2}, let $\Ga$ denote a tight graph with classical parameters with $q=1$. In this subsection, we show that $\Ga$ does not support a uniform structure.

\begin{theorem}\label{ttight}
With reference to Notation~\ref{not3.2}, let $\Ga$ denote a tight graph with classical parameters with $q=1$. Then, $\Ga$ does not support a uniform structure with respect to $x$.
\end{theorem}

\begin{proof}
Since $\Ga$ is tight, we have that $a_D=0$ by Theorem \ref{tight}(iii), which implies that $\beta= 1+\alpha(D-1)$. Furthermore, the eigenvalue $\theta_i$ of $\Ga$ is equal to $(D-i)(\beta-\alpha i)-i$ $(0 \le i \le D)$, see \cite[Corollary 8.4.2]{BCN}. By Theorem \ref{th2.3}, $\Gamma$ has, up to isomorphism, exactly two irreducible $T$-modules $W$ and $W'$ with endpoint $1$. Moreover, both of them have diameter $D-2$. One of this two modules, say $W$, has local eigenvalue $\tilde{\theta_1}$, while $W'$ has local eigenvalue $\tilde{\theta_D}$. Let $\{w_i\}^{D-2}_{i=0}$ and $\{w'_i\}^{D-2}_{i=0}$ be the standard bases for $W$ and $W'$, respectively. Let $\beta_i$, $\beta'_i$, $\gamma_i$, $\gamma'_i$ be scalars as in Proposition~\ref{tilde}. By \cite[Theorem 10.6]{GoTer}, we have that $\gamma_i = \gamma'_i= c_{i+1}$. Using \cite[Definition 4.5, Theorem 10.6]{GoTer}, we can derive that 
\begin{equation}\label{beta}
\beta_1=b_2 \hbox{ \, and \, } \beta'_1=b_2\frac{(1+\alpha)(1+\alpha(D-2))}{1+\alpha(D-3)}.
\end{equation}
Assume that $W$ and $W'$ are isomorphic as $T_f$-modules. Then, $\beta_1/\beta'_1=\gamma'_0/\gamma_0$ by Proposition \ref{tilde}. As $\gamma_0=\gamma'_0$ by the comments above, this implies $\beta'_1=\beta_1$. It follows from \eqref{beta} that $\alpha=0$ or $\alpha=-2/(D-2)$. If $\alpha=0$, then $\beta=1$, and so $\Ga$ is bipartite, a contradiction. If $\alpha=-2/(D-2)$, then $c_{D-1}=1-D<0$, a contradiction. This shows that $W$ and $W'$ cannot be isomorphic as $T_f$-modules, and so $\Ga$ does not support a uniform structure by Theorem \ref{oldpaper}(ii).
\end{proof}

\begin{corollary}\label{tttight}
With reference to Notation~\ref{not3.2}, let $\Ga$ be one of the following graphs: Johnson graph $J(2D,D)$, Gosset graph, Halved cube $\frac{1}{2}H(n,2)$ with $n$ even. Then, $\Ga$ does not support a uniform structure with respect to $x$.
\end{corollary}

\begin{proof}
Note that $\Ga$ is tight, see \cite[Section 13]{JKT} (observe that the Gosset graph belongs to the family of Taylor graphs). The result follows from Theorem \ref{ttight}.
\end{proof}

\subsection{Johnson graphs}

With reference to Notation~\ref{not3.2}, let $\Gamma=J(n,D)$ with $n\geq2D$. In this subsection, we show that $\Gamma$ does not support a uniform structure with respect to $x$. Observe that by \cite[Example 6.1]{Terpart3} every irreducible $T$-module is thin. Let $W$ denote an irreducible $T$-module with endpoint $r$, dual endpoint $t$ and diameter $d$. By  \cite[Lemma 4.5]{GZH}, the isomorphism class of $W$ is determined by $r,t$ and $d$. 

\begin{theorem}\label{johnson}
With reference to Notation~\ref{not3.2}, let $\Gamma=J(n,D)$ with $n\geq2D$. Then, $\Gamma$ does not support a uniform structure.
\end{theorem}
\begin{proof}
If $n=2D$, then the result follows from Corollary \ref{tttight}. Thus, assume that $n>2D$.  From \cite[Theorem 11.12]{MacMik} we have that $\Gamma$ has, up to isomorphism, exactly 3 irreducible $T$-modules with endpoint 1. Since the isomorphism classes of these modules are determined by their respective dual endpoint and diameter, it follows from Propositions \ref{prop:diam} and \ref{prop:dualendpt} that two of these modules have diameter $D-2$ (and dual endpoints 1 and 2, respectively), while the third module has diameter $D-1$ (and dual endpoint 1). 

Let $W$ and $W'$ denote the irreducible $T$-modules with endpoint $1$, diameter $d=D-2$ and dual endpoints $1$ and $2$, respectively. Let $\{w_i\}_{i=0}^d$ and $\{w'_i\}_{i=0}^d$ denote the standard bases for $W$ and $W'$, respectively, and let $\beta_i$, $\beta'_i$, $\gamma_i$, $\gamma'_i$ be scalars as in Proposition~\ref{tilde}. It follows from ~\cite[Example 6.1]{Terpart3} that $\beta_i=b_{i-1}(W)=(D-i-1)(n-D-i)$, $\gamma_i=c_{i+1}(W)=(i+1)(i+2)$, $\beta'_i=b_{i-1}(W')=(D-i-1)(n-D-i-1)$, and $\gamma'_i=c_{i+1}(W')=(i+1)^2$. Observe that $\beta_1/\beta'_1=(n-D-1)/(n-D-2)$ and $\gamma'_0/\gamma_0=1/2$. If $\beta_1/\beta'_1=\gamma'_0/\gamma_0$, then we would have that $n=D$, contradicting our assumption that $n>2D$. Proposition~\ref{tilde} yields that $W$ and $W'$ are not isomorphic as $T_f$-modules, and so $J(n,D)$ does not support a uniform structure by Theorem~\ref{oldpaper}(ii).
\end{proof}

\subsection{Hamming graphs}

Recall that the Hamming graph $H(D,n)$ is bipartite if and only if $n=2$. Hence, we will assume that $n \ge 3$. 

\begin{theorem}\label{thm:Hamm}
With reference to Notation~\ref{not3.2}, let $\Ga$ denote the Hamming graph $H(D,n)$ with $n\geq3$. Then, 
\begin{equation}\label{eq9}
-\frac{1}{2}RL^2+LRL-\frac{1}{2}L^2R=(n-1)L
\end{equation}
is satisfied on $\e_iV$ for $1 \le i \le D$.
In particular, $\Ga$ supports a strongly uniform structure with respect to $x$, where $e_i^{-}=-\frac{1}{2}\, (2\leq i\leq D)$, $e_i^{+}=-\frac{1}{2}\, (1\leq i\leq D-1)$, and $f_i=n-1\,(1\leq i\leq D)$.  
\end{theorem}
\begin{proof}
First note that by \cite[Example 6.1]{Terpart3}, every irreducible $T$-module is thin. Let $W$ denote an irreducible $T$-module with endpoint $r$ and diameter $d$. Let us check that \eqref{eq9} holds on $\e_{r+i}W$ ($0\le i \le d$).  Let $\{w_i\}^d_{i=0}$ denote the standard basis for $W$ (with $w_i\in \e_{r+i}W$). Let $\beta_i$ $(1\leq i \leq d)$ and $\gamma_i$ $(0\leq i \leq d-1)$ be scalars such that $Lw_i=\beta_i w_{i-1}$ and $Rw_i=\gamma_i w_{i+1}$. From \cite[Example 6.1]{Terpart3}, we have that $\beta_i=b_{i-1}(W)=(n-1)(d-i+1)$ and $\gamma_i=c_{i+1}(W)=i+1$. Observe that, if we apply  \eqref{eq9}  to $w_0$, then both sides of \eqref{eq9} are equal to zero by \eqref{eq:LRaction}.
For $1\leq i\leq d$, a straightforward computation shows that the left-hand side and the right-hand side of \eqref{eq9} agree on $w_i$. As $V$ is an orthogonal direct sum of irreducible $T$-modules, this shows that \eqref{eq9} is satisfied on $\e_i V$ for $1 \le i \le D$.

Recall the scalars $e_i^-$, $e_i^+$, and let  $U$ be the corresponding parameter matrix from Definition~\ref{3diag}. Clearly $e_i^{-}\neq 0 \; (2 \le i \le D)$ and $e_i^{+}\neq 0 \; (1 \le i \le D-1)$. For $1\leq s\leq t\leq D$, let $P=(e_{ij})_{s\leq i,j \leq t}$ denote the principal submatrix of $U$. A simple linear recurrence argument shows that $\det(P)=\frac{t-s+2}{2^{t-s+1}}$, implying that $P$ is non-singular. In virtue of Proposition~\ref{ortho}, $\Ga$ supports a strongly uniform structure with respect to $x$.   
\end{proof}

\subsection{Halved Cubes}

With reference to Notation~\ref{not3.2}, let $\Ga$ denote the Halved cube $\frac{1}{2}H(n,2)$. Recall that by Corollary~\ref{tttight} $\Ga$ does not support a uniform structure with respect to $x$ if $n$ is even. In this subsection, we show that $\Ga$ does support a (strongly) uniform structure with respect to $x$ if $n$ is odd. Observe that the diameter of $\Ga$ is $\lfloor \frac{n}{2}\rfloor$, see  \cite[p. 264]{BCN}.

\begin{theorem}\label{thm:H-odd}
	With reference to Notation~\ref{not3.2}, let $\Ga$ denote the Halved cube $\frac{1}{2}H(n,2)$ with $n$ odd, $n\geq7$. Recall that $D=\lfloor \frac{n}{2}\rfloor=(n-1)/2$. Then, 
	\begin{equation}\label{eq10}
		e_i^{-}RL^2+LRL+e_{i}^{+}L^2R=f_iL
	\end{equation}
	is satisfied on $\e_iV$ for $1 \le i \le D$, where 
	\begin{align}\label{halved}
		&e_i^{-}=\frac{4i-1-2D}{6-8i+4D}\,\, (2\leq i \leq D)\qquad e_i^{+}=\frac{4i-5-2D}{6-8i+4D}\,\, (1\leq i \leq D-1)\\&f_i=-(4i-5)(4i-1)+(16i-12)D-4D^2\,\,\, (1\leq i \leq D).
	\end{align}
	In particular, $\Ga$ supports a strongly uniform structure with respect to $x$. 
\end{theorem}
\begin{proof}
	First note that by \cite[Example 6.1 (16)]{Terpart3}, every irreducible $T$-module is thin. Let $W$ denote an irreducible $T$-module with endpoint $r$ and diameter $d$. We have that $r=(D-d-e)/2$ where $e=0$ if $D-d$ is even, and $e=-1$ if $D-d$ is odd, see \cite[Example 6.1 (16)]{Terpart3}. Let us check that \eqref{eq10} holds on $\e_{r+i}W$ ($0\le i \le d$).  Let $\{w_i\}^d_{i=0}$ denote the standard basis for $W$ (with $w_i\in \e_{r+i}W$). Let $\beta_i$ $(1\leq i \leq d)$ and $\gamma_i$ $(0\leq i \leq d-1)$ be scalars such that $Lw_i=\beta_i w_{i-1}$ and $Rw_i=\gamma_i w_{i+1}$. From \cite[Example 6.1 (16)]{Terpart3}, we have that 
\begin{align*}  
\beta_i&=b_{i-1}(W)=\begin{cases}
(d-i+1)(2d-2i+3) \qquad \hbox{if $D-d$ is even},\\
(d-i+1)(2d-2i+1) \qquad \hbox{if $D-d$ is odd},
\end{cases} \\[.2in] 
\gamma_i&=c_{i+1}(W) =\begin{cases}
(i+1)(2i+1) \qquad \hbox{if $D-d$ is even},\\
(i+1)(2i+3) \qquad \hbox{if $D-d$ is odd}.
\end{cases}
\end{align*}
Observe that, if we apply  \eqref{eq10}  to $w_0$, then both sides of \eqref{eq10} are equal to zero by \eqref{eq:LRaction}.
	For $1\leq i\leq d$, a straightforward but somehow tedious computation shows that the left-hand side  and the right-hand side of \eqref{eq10} both agree on $w_i$ (note that $w_i \in \e_{r+i}V$, and so when \eqref{eq10} is applied to $w_i$, the subscripts in \eqref{eq10} must be equal to $r+i$). As $V$ is an orthogonal direct sum of irreducible $T$-modules, this shows that \eqref{eq10} is satisfied on $\e_i V$ for $1 \le i \le D$.
	
	Recall the scalars $e_i^-$ and $e_i^+$ from \eqref{halved}, and let $U$ denote the corresponding parameter matrix from Definition~\ref{3diag}. Let us first argue that the denominators in \eqref{halved} are non-zero. Indeed, as $6$ is not divisible by $4$, we have that $6-8i+4D$ is non-zero for $1\leq i \leq D$. Note also that $4i-1-2D$ and $4i-5-2D$ are odd integers for $1\leq i \leq D$, implying that  $e_i^{-}\neq 0 \; (2 \le i \le D)$ and $e_i^{+}\neq 0 \; (1 \le i \le D-1)$. For $1\leq s\leq t\leq D$, let $U_{s,t}=(e_{ij})_{s\leq i,j \leq t}$ denote the principal submatrix of $U$. It is easy to see that if $s=t$ then $\det(U_{s,t})=1$. Similarly, if $t=s+1$, then 
$$
  \det(U_{s,t})=1-e_s^+ e_{s+1}^- = \frac{3(2D-2t-2s+3)(2D-4t+5)}{4(2D-4t+3)(2D-4t+7)}.
$$
If $t \ge s+2$, expanding $\det(U_{s,t})$ by the first row and then by the first column, we get that
$$
  \det(U_{s,t}) = \det(U_{s+1,t}) - e_s^+ e_{s+1}^- \det(U_{s+2,t}).
$$
A simple induction argument shows that 
\begin{equation} \label{halved:det}
\det(U_{s,t})=\frac{(t-s+2)(2D-2t-2s+3)\displaystyle{\prod^{t-s-1}_{i=0}(2D-4t+5+4i)}}{2^{t-s+1}\displaystyle{\prod^{t-s}_{i=0}(2D-4t+3+4i)}}.
\end{equation}  
Observe that all factors in the numerator of \eqref{halved:det} are nonzero, implying that $U_{s,t}$ is non-singular. In virtue of Proposition~\ref{ortho}, $\Ga$ supports a strongly uniform structure with respect to $x$.   
\end{proof}

\subsection{Doob graphs}

Recall that the Doob graph $D(n,m)$, $n\geq1$, $m\geq0$, is a Cartesian product of $n$ copies of the Shrikhande graph and $m$ copies of the complete graph $K_4$, see \cite[p. 262]{BCN}. With reference to Notation~\ref{not3.2}, let $\Ga$ denote the Doob graph $D(n,m)$. In this subsection, we show that $\Ga$ supports a strongly uniform structure with respect to $x$. Observe that the diameter of $\Ga$ is $2n+m$, see  \cite[p. 262]{BCN}.

\begin{theorem}\label{thm:Doob}
With reference to Notation~\ref{not3.2}, let $\Ga$ denote the Doob graph $D(n,m)$ with $n\geq1$, $m\geq0$ and $D=2n+m\geq3$. Then, 
\begin{equation}\label{eq11}
-\frac{1}{2}RL^2+LRL-\frac{1}{2}L^2R=3L
\end{equation}
is satisfied on $\e_iV$ for $1 \le i \le D$.
In particular, $\Ga$ supports a strongly uniform structure with respect to $x$, where $e_i^{-}=-\frac{1}{2}\, (2\leq i\leq D)$, $e_i^{+}=-\frac{1}{2}\, (1\leq i\leq D-1)$, and $f_i=3\,(1\leq i\leq D)$.  
\end{theorem}

\begin{proof}
Let $W$ denote an irreducible $T$-module with endpoint $r$. Then, by \cite[Theorem 1, Proposition 3]{terrible}, there exist certain non-negative integers $\delta$, $p$ (subject to certain restrictions given in \cite[Theorem 1]{terrible}), such that $W$ has basis $\{w_{\ell,j} \; | \; 0\leq \ell \leq \delta, \; 0\leq j \leq p\}$ with $w_{\ell,j} \in \e_{r+\ell+j}W$. The action of $L$ and $R$ on this basis is given by 
\begin{equation}\label{doobL}
Lw_{\ell,j}=3(\delta-\ell+1)w_{\ell-1,j}+(p-j+1)w_{\ell,j-1},
\end{equation} 
and
\begin{equation}\label{doobR}
Rw_{\ell,j}=3(j+1)w_{\ell,j+1}+(\ell+1)w_{\ell+1,j},
\end{equation}
where $w_{\ell,j}$ is defined to be zero if $(\ell,j) \notin \{0,\ldots, \delta\}\times\{0,\ldots,p\}$, see \cite[Proposition 3]{terrible}. Note that if $(\ell,j) \notin \{0,\ldots, \delta\}\times\{0,\ldots,p\}$, equations \eqref{doobL} and \eqref{doobR} still hold.  We claim that the left-hand side and the right-hand side of \eqref{eq11} agree when applied to $w_{\ell,j}$. Using \eqref{doobL} and \eqref{doobR}, straightforward calculations show that
\begin{eqnarray}
RL^2w_{\ell,j}&=&27(\delta-\ell+1)(\delta-\ell+2)(j+1)w_{\ell-2,j+1}+ \\ \nonumber 
&&9(\delta-\ell+1)((\delta-\ell+2)(\ell-1)+2j(p-j+1))w_{\ell-1,j}+\\ \nonumber
&&3(p-j+1)(2\ell(\delta-\ell+1)+(j-1)(p-j+2))w_{\ell,j-1}+\\ \nonumber
&&(p-j+1)(p-j+2)(\ell+1)w_{\ell+1,j-2}; \nonumber
\end{eqnarray}
\begin{eqnarray}
LRLw_{\ell,j}&=&27(\delta-\ell+1)(\delta-\ell+2)(j+1)w_{\ell-2,j+1}+\\ \nonumber 
&&9(\delta-\ell+1)(\ell(\delta-\ell+1)+2j(p-j)+p)w_{\ell-1,j}+\\ \nonumber
&&3(p-j+1)(\delta(2\ell+1)-2\ell^2-j(j-p-1))w_{\ell,j-1}+\\ \nonumber
&&(p-j+1)(p-j+2)(\ell+1)w_{\ell+1,j-2};
\end{eqnarray}
\begin{eqnarray}
L^2Rw_{\ell,j}&=&27(\delta-\ell+1)(\delta-\ell+2)(j+1)w_{\ell-2,j+1}+\\ \nonumber 
&&9(\delta-\ell+1)((\delta-\ell)(\ell+1)+2(j+1)(p-j))w_{\ell-1,j}+\\ \nonumber
&&3(p-j+1)(2(\delta-\ell)(\ell+1)+(j+1)(p-j))w_{\ell,j-1}+\\ \nonumber
&&(p-j+1)(p-j+2)(\ell+1)w_{\ell+1,j-2}.
\end{eqnarray}
It follows from the above computations that the claim is true.  As $V$ is an orthogonal direct sum of irreducible $T$-modules, this shows that \eqref{eq11} is satisfied on $\e_i V$ for $1 \le i \le D$.

Recall the scalars $e_i^-$, $e_i^+$, and let $U$ be the corresponding parameter matrix from Definition~\ref{3diag}. Clearly  $e_i^{-}\neq 0 \; (2 \le i \le D)$ and $e_i^{+}\neq 0 \; (1 \le i \le D-1)$. For $1\leq s\leq t\leq D$, let $P=(e_{ij})_{s\leq i,j \leq t}$ denote the principal submatrix of $U$. A simple linear recurrence argument shows that $\det(P)=\frac{t-s+2}{2^{t-s+1}}$, implying that $P$ is non-singular. In virtue of Proposition~\ref{ortho}, $\Ga$ supports a strongly uniform structure with respect to $x$.  

\end{proof}

\begin{remark}
Let $G$ and $H$ be arbitrary connected graphs. Fix a vertex $x$ of $G$ and a vertex $y$ of $H$. Using the definition of adjacency in the Cartesian product $G \square H$, one can show that 
$$
(G\square H)_f((x,y))=G_f(x)\square H_f(y).
$$
Let $S$ denote the Shrikhande graph, and fix a vertex $x$ of $S$. Similarly, fix a vertex $y$ of the Cartesian product $K_4\square K_4$. It is not complicated to check that the graphs $S_f(x)$ and $(K_4\square K_4)_f(y)$ are isomorphic. 

Recall our Doob graph $D(n,m)$, and fix a vertex $x$ of $D(n,m)$. Furthermore, fix a vertex $y$ of the Hamming graph $H(2n+m, 4)$. Using the above observations, it easily follows that 
$$
D(n,m)_f(x) \cong H(2n+m, 4)_f(y),
$$
which, combined with Theorem \ref{thm:Hamm}, provides an alternative proof of Theorem \ref{thm:Doob}.
\end{remark}

\noindent
We finish the paper by summarizing the results of Section \ref{sec:q=1}.

\begin{theorem}\label{thm:main}
With reference to Notation~\ref{not3.2}, assume that $\Ga$ has classical parameters $(D, 1, \alpha, \beta)$. Then, $\Ga$ supports a uniform structure with respect to $x$ if and only if $\Ga$ is one of the following graphs:
\begin{enumerate}[label=(\roman*)]
		\item Hamming graph $H(D,n)$, $D\ge 3$;
    	\item Halved cube  $\frac{1}{2}H(n,2)$ with $n \ge 7$ odd;
    	\item Doob graph $D(n,m)$ with $n \ge 1$, $m \ge 0$ and $2n+m \ge 3$.
\end{enumerate}
Moreover, the above graphs support a strongly uniform structure with respect to $x$.
\end{theorem}
\begin{proof}
	Immediately follows from Theorem \ref{thm:Hamm}, Theorem \ref{thm:H-odd}, and Theorem \ref{thm:Doob}.
\end{proof}

\subsection*{Acknowledgment}
Blas Fernández's work is supported in part by the Slovenian Research Agency (research program P1-0285, research projects J1-2451, J1-3001 and J1-4008). \v{S}tefko  Miklavi\v{c}'s  research is supported in part by the Slovenian Research Agency (research program P1-0285 and research projects J1-1695, N1-0140, N1-0159, J1-2451, N1-0208, J1-3001, J1-3003, J1-4008 and J1-4084).
Roghayeh Maleki and Giusy Monzillo's research are supported in part by the Ministry of Education, Science and Sport of Republic of Slovenia (University of Primorska Developmental funding pillar).


 \end{document}